\documentclass[10pt]{article}
\usepackage{amssymb,amsmath,amsfonts,amsthm,enumerate,color}
\usepackage[toc,page,title,titletoc,header]{appendix}
\usepackage{graphicx}
\usepackage{indentfirst}
\usepackage{multicol}
\usepackage{booktabs}
\usepackage[T1]{fontenc}

\numberwithin{equation}{section}

\newtheorem{Theorem}{Theorem}[section]
\newtheorem{Lemma}[Theorem]{Lemma}

\newtheorem{Definition}[Theorem]{Definition}

\newtheorem{Remark}[Theorem]{Remark}
\numberwithin{equation}{section}

 \def\p{\partial} 
\def \Vh0{\stackrel{\circ}{V}_h} \def\to{\rightarrow}
   
\def\Om{\Omega}  
  
\newcommand{\q}{\quad}    \def\R{{\mathbb R}}
  \def\f{\frac}  

\def\m{\mbox} \def\t{\times}  
  
\def\o{\overline }
\def\p{\partial}

\newcommand{\lc}
{\mathrel{\raise2pt\hbox{${\mathop<\limits_{\raise1pt\hbox
{\mbox{$\sim$}}}}$}}}

\newcommand{\gc}
{\mathrel{\raise2pt\hbox{${\mathop>\limits_{\raise1pt\hbox{\mbox{$\sim$}}}}$}}}

\newcommand{\ec}
{\mathrel{\raise2pt\hbox{${\mathop=\limits_{\raise1pt\hbox{\mbox{$\sim$}}}}$}}}

\def\bb{\begin{equation}} \def\ee{\end{equation}}

\def\beqn{\begin{eqnarray}}  \def\eqn{\end{eqnarray}}

\def\beqnx{\begin{eqnarray*}} \def\eqnx{\end{eqnarray*}}

\def\bn{\begin{enumerate}} \def\en{\end{enumerate}}

\def\bd{\begin{description}} \def\ed{\end{description}}



\title{
Well-posedness of a Pulsed Electric Field Model \\
in Biological Media and its Finite Element Approximation}
\author{
Habib Ammari\footnote{Department of Mathematics and Applications, Ecole Normale Sup$\acute{\text{e}}$rieure, 45 Rue d'Ulm, 75005 Paris, France. The work of this author
was supported by ERC Advanced Grant Project MULTIMOD--267184.
(habib.ammari@ens.fr).}
\and Dehan Chen\footnote{Department of Mathematics, Chinese University of Hong Kong, Shatin, N.T., Hong Kong (dhchen@math.cuhk.edu.hk).}
\and Jun Zou\footnote{Department of Mathematics, Chinese University of Hong Kong, Shatin, N.T., Hong Kong.
The work of this author was substantially supported by Hong Kong RGC grants (projects 405513 and 404611).
(zou@math.cuhk.edu.hk).}
}

\begin{document}

\date{}
\maketitle

\begin{abstract}
This work aims at providing a mathematical and numerical framework for the analysis on
the effects of pulsed electric fields on biological media.  Biological tissues and cell suspensions are described as having a heterogeneous permittivity and a heterogeneous conductivity.
Well-posedness of the model problem and the regularity of its solution are established.
A fully discrete finite element scheme
is proposed for the numerical approximation of the potential distribution as a function of time and space simultaneously for an arbitrary shaped pulse, and it is demonstrated to enjoy the optimal convergence order in both space and time.
The presented results and numerical scheme have potential applications in the fields of medicine, food sciences, and biotechnology.

\end{abstract}

\bigskip



\section{Introduction}

The electrical properties of biological tissues and cell suspensions determine the pathways of current flow through the medium and, thus, are very important in the analysis of a wide range of biomedical applications and in food sciences and biotechnology \cite{ food, biotech, blood}.

A biological tissue is described as having a permittivity and a conductivity \cite{bio3}. The conductivity can be regarded as a measure of the ability of its charge to be transported throughout its volume by an applied electric field while the permittivity is a measure of the ability of its dipoles to rotate or its charge to be stored by an applied external field. At low frequencies, biological tissues behave like a conductor but capacitive effects become important at higher frequencies due to the membranous structures \cite{schwan3, jinkeun}.

In this paper, we consider
a model problem for the effect of pulsed electric fields on biological tissues.
Our goal is to study the electric behavior of a biological tissue under the influence of
a pulsed electric field. It is of great importance to understand the effects of the pulse shape on the potential distribution in the tissue medium. We provide a numerical scheme for computing the potential distribution as a function of time and space simultaneously for an arbitrary shaped pulse. Our results are expected to have important applications in neural activation during deep brain simulations \cite{bio1, bio2}, debacterization of liquids, food processing \cite{food2}, and  biofouling prevention \cite{bio4}. Our numerical scheme can be also used for
selective spectroscopic imaging of the electrical properties of biological media \cite{laure}.
It is challenging to specify the pulse shape in order to give rise to selective imaging of cell suspensions \cite{pulse1,pulse2}.

The paper is organized as follows. In section \ref{sec2} we introduce the model equation and some notations and preliminary results. We recall the method of continuity and the notions of weak and strong solutions. Section \ref{sec3} is devoted to existence, uniqueness, and regularity results for the solution to the model problem. We first derive an a priori energy estimate. Then we prove existence and uniqueness of the weak solution. Finally, we investigate the interface problem where the conductivity and permittivity distributions may be discontinuous, which is a common feature of biological media.  It is shown in section \ref{sec3} that the solution to the interface problem has a higher regularity in each individual region than in the entire domain. This  regularity result is critical for our further numerical analysis.
In section \ref{sec4} we investigate the numerical approximation of the solution to the interface problem. Assuming that the domain is a convex polygon, we present a semi-discrete scheme and prove the error estimates for it in both $H^1$-  and $L^2$-norms. With these estimates at hand,
we then process to propose a fully-discrete scheme and establish the error estimates for it in both $H^1$- and $L^2$-norms. It is worth mentioning that both
semi-discrete and fully discrete scheme achieve optimal convergence order in both
$H^1$- and $L^2$-norm, provided that the interface is exactly resolved.

Let us end this section with some notations used in this paper. For
a domain $U\subset  \mathbb{R}^n, n=2$ or $3$, each integer $k\geq 0$ and real $p$ with
$1\leq p\leq \infty$, $W^{k,p}(U)$ denotes the standard Sobolev space of functions with their weak derivatives
 of order up to $k$ in the Lebesgue space $L^p(U)$. When $p=2$, we write $H^k(U)$ for $W^{k,2}(U)$.
The scalar product of $L^2(\Omega)$ is denoted by $(\cdot,\cdot)$.
If $X$ is a Banach space with norm $\|\cdot\|_X$
and $J\subset \mathbb{R}$ is an interval, then $L^2(J;X)$ represents
the Banach space consisting of all quadratically integrable functions $f:J\to X$
(in B\^ocher sense) with norm:
$\|f(t)\|_{L^2(J;X)}:=\left(\int_J\|f(t)\|_X^2dt\right)^{1/2}$.
We denote by $H^1(J;X)$ the space of all functions $u\in L^2(J;X)$ such that
$u'$, the weak derivative of  $u$ with respect to time variable, exists and belongs to $L^2(J;X)$, endowed with the norm
$\|u\|_{H^1(J;X)}=\left(\|u\|^2_{L^2(J;X)}+\|u'\|^2_{L^2(J;X)}\right)^{1/2}$.
For $1\leq i,j\leq n$, we write $D_iu=\p u/\p x_i$ and $D_{i,j}u=\p^2 u/\p x_i\p x_j$.
For $u\in H^1(U)$ and $f\in H^1(J;H^1(U))$,
we also set the semi-norms $|u|_{H^1(U)}:=\|\nabla u\|_{L^2(U)}$ and
$|f|_{L^2(J;H^1(U))}:=(\int_{J}|f(t)|^2_{H^1(U)}dt)^{\f{1}{2}}$.
For ease of notation, we do not always distinguish between the notation of 
$u$, $u(t)$, $u(t,x)$ and $u(t,\cdot)$. Sometimes, the notation is not changed when a 
function  defined on $\Omega$ restricted to a subset. 
For the sake of brevity, we systematically use the expression
$A\lesssim B$ to indicate that $A\leq C B$ for constant $C$ that is independent of
 $A$ and $B$. In some special cases, we may specify the used constants.

\section{Preliminary} \label{sec2}

Let $\Omega$ be a bounded domain with Lipschitz boundary. Let $\sigma$ and $\varepsilon$
denote the conductivity and permittivity distributions inside $\Omega$. We assume that
$\sigma$ and $\varepsilon$  belong to $L^\infty(\Omega)$. Biological tissues induce capacitive effects due to their cell membrane structures \cite{bio3}. When they are exposed to electric
pulses, the voltage potential $u$ is a solution to the following time-dependent equation \cite{addref1, addref2}
\bb\label{Equ}
\left\{\begin{array}{rlll}
-\nabla\cdot\left(\sigma(x)\nabla u(t,x)+\varepsilon(x)\nabla u'(t,x)\right)&=&f(t,x), \quad (t,x)\in (0,T)\times \Omega, \\
 u&=&0 ,\quad (t,x) \in (0,T)\times\partial\Omega,\\
u(0,x) &=&u_0,\quad  x\in\Omega,
\end{array}\right.
\ee
where $u_0$ is the initial voltage and $T$ is the final observation time and $f \in L^2(]0,T[;H^{-1}(\Omega))$ is the
electric pulse.

The goal of this work is to establish the well-posedness of the model system (\ref{Equ}) and derive
a fully discrete finite element scheme for the numerical solution of the system.
Of our special interest is the case when the physical coefficients are discontinuous in $\Omega$, namely
they may have large jumps across the interface between two different media,
which is a common feature in applications, and the conductivity distribution $\sigma(x)$ does not need to be bounded 
below strictly positively. 
As far as we know, this is the first mathematical and numerical work on pulsed electrical fields in capacitive media.
The main difficulty comes from the fact that (\ref{Equ}) does not belong to the well-studied classes of time-dependent
equations. Our results in this paper have potential applications in cell electrofusion and electroporation using eletric pulses \cite{addref2} and in electrosensing \cite{fish1}.

In this section, we first introduce some notions and preliminary results.
For the sake of brevity, we write $I=]0,T[$, $H=L^2(\Omega)$, $V=H^1_0(\Omega)$ with its dual space $V'=H^{-1}(\Omega)$
and $\mathcal{X}=H^1_0(\Omega)\cap H^2(\Omega)$.
 Clearly,
$V\subset H\subset V'$ is a triple of spaces (cf. \cite[Chapter 1]{Yagi}), i.e.,
\begin{enumerate}
\item[$(1)$]\ the embeddings $V\subset H\subset V'$ are dense and continuous;

\item[$(2)$]\ $\{V',V\}$ forms an adjoint pair with duality product $\langle \cdot,\cdot\rangle_{V'\times V}$;

\item[$(3)$]\  the duality product  $ \langle \cdot,\cdot \rangle_{V'\times V}$ satisfies
$$
\langle u,v\rangle_{V'\times V}=(u,v),  \quad\ \forall \ u\in H, \ v\in V.
$$
\end{enumerate}
We also introduce two bilinear forms $a_1(u,v)$ and $a_2(u,v)$  on $V$ as follows:
\begin{equation}
a_1(u,v)=\int_{\Omega}\sigma(x)\nabla u(x)\cdot \nabla v(x)dx,
\quad a_2(u,v)=\int_{\Omega}\varepsilon(x)\nabla u(x)\cdot \nabla v(x)dx,
\quad u,v\in V.
\end{equation}

We first define the weak and strong solutions to the equation (\ref{Equ}).
We  adapt the widely used notions of weak and strong solutions of parabolic equations
(see, for instance, \cite{Robinson}).
\begin{Definition}\label{Def-weak}
Let $u_0\in V$ and $f\in L^2(I;V')$. A function $u\in H^1(I;V)$ is called
a week solution of  {$(\ref{Equ})$} if $u(0)=u_0$ and it satisfies the following
weak formulation:
\begin{equation}\label{Def}
a_1(u(t,\cdot),v)+a_2(u'(t,\cdot),v)=\langle f(t,\cdot),v\rangle_{V'\times V}
\end{equation}
for all $v\in H^1_0(\Omega)$ and a.e. $t\in I$.
\end{Definition}

\begin{Definition} \label{Def2}
Let $f\in L^2(I;H)$ and $u_0\in \mathcal{X}$.
Then, a function
$u\in H^1(I;\mathcal{X})$ is called a strong solution of (\ref{Equ}) if $u(0)=u_0$ and
the relation
\begin{equation}
-\nabla\cdot\left(\sigma(x) \nabla u(t,x)+\varepsilon(x) \nabla u'(t,x)\right)=f(t,x)
\end{equation}
holds for a.e. $t\in I$ and a.e. $x\in \Omega$.
\end{Definition}
\begin{Remark}
Let $X$ be a Banach space. From \cite[Proposition\,7.1]{Robinson}
we know that $H^1(I;X)\Subset C(\overline{I};X)$
continuously and
\bb\label{embeding}
\sup_{t\in\overline{I}}\|u(t)\|\lesssim
\|u\|_{H^1(I;X)}.
\ee
In particular, we have that $u\in C(\o{I};V)$ for $u\in H^1(I;V)$.
\end{Remark}

To prove the existence below, we will use the so-called ``method of continuity'', whose key tool is the following
lemma (e.g., \cite{Gilbarg}).
\begin{Lemma}\label{Continuity}
Let $X$ be a Banach space, $Y$ a normed linear
space, and  $L_0, L_1$ two bounded linear operators from $X$
to $Y$. For each $\lambda\in[0,1]$, set
$$
L_\lambda=(1-\lambda)L_0+\lambda L_1,
$$
and suppose that there exists a constant $C$ such that
$$
\|x\|_{X}\leq C\|L_\lambda x\|_{Y},  \quad \forall\ x\in X, \quad \lambda\in[0,1].
$$
Then $L_1$ maps $X$ onto  $Y$
if and only if $L_0$ maps  $X$ onto  $Y$.
\end{Lemma}

Let $u$ be a function in a domain $U\subset\R^n$, $W\Subset U$ and $e_k$
the unit coordinate vector in the $x_k$ direction. We define
the difference quotient of $u$ in the direction $e_k$ by
\bb\label{DQ}
D_k^hu(x)=\f{u(x+he_k)-u(x)}{h}
\ee
for $x\in W$ and $h\in\R$ with $0<|h|< \mathrm{dist}(W,\p U)$.
We will use the following lemma in
the proof of Theorem \ref{Regularity}, concerning the
difference quotient of functions in Sobolev spaces (cf. \cite[Lemma 7.23]{Gilbarg}).
\begin{Lemma}\label{DQ1}
Suppose that $u\in H^1(U)$. Then for each
$W\Subset U$,
$$
\|D_k^hu\|_{L^2(W)}\leq \|D_k u\|_{L^2(U)} , \, \q \forall\, h: ~
0<|h|<\f{1}{2} \mathrm{dist}(W,\p U)\,.
$$
\end{Lemma}
We end up with an analogue of \cite[Lemma 7.24]{Gilbarg}, and provide 
an outline of its proof.  
\begin{Lemma}\label{DQ3}
Let $u\in L^2(I;L^2(U))$, $W\Subset U$ and suppose that there exists a positive constant $K$ such
that $\|D_k^hu\|_{L^2(I;L^2(W))}\leq K$ for all $0<|h|<\f{1}{2} \mathrm{dist}(W,\p U)$.
Then $\|D_k u\|_{L^2(I;L^2(W))}\leq K$.
\end{Lemma}
\begin{proof}
 Banach-Alaoglu theorem implies that there exists a sequence
 $\{h_m\}^\infty_{m=1}$ with $h_m\to 0$ and a function $v\in L^2(I;L^2(W))$
 such that $\|v\|_{L^2(I;L^2(W))}\leq K$, and  for any $\varphi\in C^\infty_0(W)$ and $\alpha\in C^\infty_0(I)$,
 $$
 \int_I\int_W \alpha(t) \varphi D_k^{h_m}u(t)dxdt  \to  \int_I\int_W \alpha(t)\varphi v(t) dxdt \q \mbox{as} ~~m\to \infty\,.  
 $$
On the other hand,  we have
 $$
 \int_I\int_W \alpha(t) \varphi D_k^{h_m}u(t) dxdt= -\int_I\int_W \alpha(t) u(t)D_k^{-h_m}\varphi dxdt  \to
  -\int_I\int_W\alpha(t) u(t)D_k\varphi dxdt,
 $$
 as $m\to \infty$.  Hence, we have
$$
\int_I\int_W \alpha(t)\left( u(t)D_k\varphi+v(t)\varphi\right)dx dt=0.
$$
Using the arbitrariness of $\alpha$  and $\varphi$, we know for a.e. $t\in I$, $v(t)=D_ku(t)$ in weak sense,
hence $v=D_ku$ in $L^2(I;L^2(W))$.
\end{proof}
\begin{Lemma}\label{tool1}
Let $U\subset \R^n$ be a domain and $1\leq i\leq n$.
If $u\in H^1(I;L^2(U))$, $D_iu'\in L^2(I;L^2(U))$ and $D_iu(0)\in L^2(U)$, then
$D_iu\in L^2(I;L^2(U))$ and
 $$
 \|D_iu\|_{L^2(I;L^2(U))}\lesssim  \|D_iu'\|_{L^2(I;L^2(U))}+\|D_iu(0)\|_{L^2(U)}.
 $$
\end{Lemma}
\begin{proof}
Since $u\in H^1(I;L^2(U))$, we have
$$
u(t)=u(0)+\int^t_0u'(s)ds,  \quad \forall\,t\in I.
$$
By Fubini's theorem, we know  that for any $\phi\in C^\infty_0(U)$,
$$
\int_{U}u(t)D_i\phi dx=\int_ U u(0)D_i\phi dx+\int_U\int^t_0 u'(s)D_i \phi dsdx= -\int_U\left(D_i u(0)+\int^t_0 D_i u'(s)ds\right)\phi dx,
$$
which implies that
$$
D_i u(t)=D_i u(0)+\int^t_0 D_i u'(s)ds.
$$
This completes the proof.
\end{proof}
\section{Existence and regularity}\label{sec3}

We now introduce a basic assumption for 
the existence and uniqueness of weak solutions to (\ref{Equ}).

\bn
\item[$(A1)$] $\sigma$, $\varepsilon \in L^{\infty}(\Omega)$, and there exist two positive constants $m$ and $M$ such that
   $ 0\leq \sigma(x)\leq M$  and $m\leq \varepsilon(x)\leq M$ for a.e. $x\in\Omega$.
\en
Let us recall that there exist
two operators $\mathcal{A}_1,\mathcal{A}_2:V\rightarrow V'$ associated with the bilinear
 forms $a_1(\cdot,\cdot)$
and $a_2(\cdot,\cdot)$, respectively, i.e.,
\begin{equation*}
\langle\mathcal{A}_1u,v\rangle_{V'\times V}=a_1(u,v),
\quad \langle\mathcal{A}_2u,v\rangle_{V'\times V}=a_2(u,v),
\quad u,v\in V.
\end{equation*}
>From  \cite[Theorem 1.24]{Yagi} we know that
$\mathcal{A}_1$ is a bounded operator and satisfies the following estimate
\begin{equation}\label{bounded2}
\|\mathcal{A}_1 u\|_{V'}\leq M\|u\|_{V}, \quad \forall u\in V,
\end{equation}
and $\mathcal{A}_2$
is  actually an isomorphism from $V$ to $V'$ and satisfies
\begin{equation}\label{bounded}
m\|u\|_{V}\leq \|\mathcal{A}_2 u\|_{V'}\leq M\|u\|_{V}, \quad \forall u\in V.
\end{equation}

\subsection{Existence and uniqueness of weak solutions}
In this subsection, we prove the existence and
uniqueness of the weak solutions to (\ref{Equ}).
The first auxiliary result is the following a prior estimate,
which lays the foundation for our subsequent existence and
regularity results of weak solutions
to (\ref{Equ}).
\begin{Theorem}\label{prior-estimate}
Let $f\in L^2(I;V')$, $u_0\in V$ and $u$  be the weak solution to $(\ref{Equ})$.
Under the assumption $(A1)$, we have
\bb\label{estimate}
\int^T_{0}\int_\Omega|\nabla u'|^2dxdt+\sup_{t\in \overline{I}}\|u(t)\|^2_{V}\lesssim \|f\|^2_{L^2(I;V')}+\|u_0\|^2_V,
\ee
and
\bb\label{estimate2}
\|u\|_{H^1(I;V)}\lesssim \|f\|_{L^2(I;V')}+\|u_0\|_V.
\ee

\end{Theorem}
\begin{proof}

Choosing $v=u'$ in  (\ref{Def}) and integrating over $(0,T)$, we obtain
\begin{equation}\label{equ2}
\int^T_{0}\int_{\Omega}\left(\sigma\nabla u(t)\cdot \nabla u'(t)+\varepsilon|\nabla u'(t)|^2\right)dxdt\leq
\int^T_{0}\|f(t)\|_{V'}\|u'(t)\|_{V}dt.
\end{equation}
>From this and the identity that
 \begin{equation*}
\int^T_{0}\int_{\Omega}\sigma \nabla u(t)\cdot \nabla u' (t)dxdt=\|\sqrt{\sigma}\nabla u(T)\|^2_{H}
-\|\sqrt{\sigma}\nabla u(0)\|^2_{H},
\end{equation*}
it follows that
\begin{equation}\label{equ2}
\int^T_{0}\int_{\Omega}\varepsilon|\nabla u'(t)|^2 dxdt\leq
\int^T_{0}\|f(t)\|_{V'}\|u'(t)\|_{V}dt+M\|u_0\|_{V}.
\end{equation}
Using Young's inequality, we have
\begin{equation*}
\|u'\|_{L^2(I;V)}\lesssim \|f\|_{L^2(I;V')}+\|u_0\|_{V}.
\end{equation*}
>From Lemma \ref{tool1} and Remark 3.2, the desired results follow immediately.
\end{proof}

With estimate (\ref{estimate2}) in hand, we can prove the first existence result of (\ref{Equ}).
\begin{Theorem}\label{existence1.5}
Let $f\in L^2(I;V')$ and $u_0\in V$. Under the assumption $(A1)$,
equation $(\ref{Equ})$
admits a unique weak solution.
\end{Theorem}
\begin{proof}
We first establish the results for $u_0=0$. 
The uniqueness is nothing but a direct consequence of Theorem \ref{prior-estimate}.
We use Lemma \ref{Continuity} to prove the existence. First, we construct a linear operator
$\mathcal{L}:H^1_0(I;V)\rightarrow L^2(I;V')$ by setting
$$
(\mathcal{L}u)(t):=\mathcal{A}_1u(t)+\mathcal{A}_2u'(t),   \quad \forall u\in H^1_0(I;V),
$$
where $H^1_0(I;V)$ is defined by
$$
H^1_0(I;V)=\{u\in H^1(I;V);u(0)=0\}.
$$
It is a closed subspace of the Banach space $H^1(I;V)$, since $H^1(I;V)\Subset C(\overline{I};V)$ continuously.
>From  (\ref{bounded2}) and (\ref{bounded}) it follows that
$$
\|\mathcal{L}u\|_{L^2(I,V')}\leq M\|u\|_{H^1(I;V)},
$$
which implies that
$\mathcal{L}$ is well-defined and continuous.

For each $\lambda\in[0,1]$, we introduce a linear operator $\mathcal{L}_\lambda: H^1_0(I;V)\to L^2(I;V')$ as follows:
$$
\mathcal{L}_\lambda u:=\lambda \mathcal{L}u+(1-\lambda)\mathcal{L}_0u ,  \quad \forall  u\in H^1_0(I;V),
$$
where we set $\mathcal{L}_0u=-\Delta u-\Delta u'$  for $u\in H^1_0(I;V)$.
Here $(-\Delta)$ is seen as an operator from $V$ to
$V'$ (cf. \cite[Theorem 2.2]{Yagi}). More precisely, it is the
operator associated with the bilinear form $a(\cdot,\cdot):V\times V\to \R$, defined by
$$
a(u,v):=\int_{\Omega}\nabla u\cdot \nabla v dx,  \quad \forall u,v\in V,
$$
in a way such that $\langle(-\Delta)u,v\rangle_{V'\times V}= a(u,v)$ for all $u,v\in V$.
In addition, $(-\Delta):V\to V'$ is an isomorphism.

Let $\sigma_\lambda =\lambda\sigma +(1-\lambda)\chi_{\Omega}$ and
$\varepsilon_\lambda=\lambda\varepsilon+ (1-\lambda)\chi_{\Omega}$
for $\lambda\in[0,1]$. Then the functions $\sigma_\lambda$ and $\varepsilon_\lambda$ satisfy
$$
m':=\min\{m,1\}\leq \varepsilon_\lambda(x)\leq M':=\max\{M,1\}
\quad \mbox{for a.e. } x \in \Omega,
$$
and
$$
0 \leq \sigma_\lambda(x)\leq M'
\quad \mbox{for a.e. } x \in \Omega.
$$
Then, for $f\in L^2(I;V')$, if $\mathcal{L}_\lambda u=f$
for some $u\in H^1_0(I;V')$, then $u(0)=0$ and for a.e. $t\in I$,
$u$ satisfies the following weak formulation:
\begin{equation}\label{Equ2}
\begin{array}{rlll}
\displaystyle \int_{\Omega}\left(\sigma_{\lambda}\nabla u(t)\cdot \nabla v
+\varepsilon_{\lambda}\nabla u'(t)\cdot \nabla v\right)
dx=\langle f(t),v\rangle_{V'\times V}, \quad \forall v\in V.
\end{array}
\end{equation}
Thus, an application of Theorem \ref{prior-estimate} yields that there exists a positive constant $C$,
depending only on $m'$, $M'$ and $T$, such that
$$
\|u\|_{H^1(I;V)}\leq C\|L_\lambda u\|_{L^2(I,V')}.
$$
In view of Lemma \ref{Continuity}, it remains
to prove that the mapping $\mathcal{L}_0:H^1_0(I;V)\to L^2(I;V')$ is onto.
To this end, for an arbitrary $f\in L^2(I;V')$, we need to construct a function $w\in H^1_0(I;V)$
such that for a.e. $t\in I$,
\begin{equation}\label{Equ3}
\begin{array}{rlll}
a(w(t),v)+a(w'(t),v)=\langle f(t),v\rangle_{V'\times V}, \quad \forall v\in V.
\end{array}
\end{equation}
Let $w(t)=\int^t_{0}e^{-t+s}h(s)ds$ for $t\in I$, where $h(s)=(-\Delta)^{-1}f(s)$ for $s\in I$.
Since $(-\Delta)^{-1}:V'\rightarrow V$ is
bounded, we have that $h\in L^2(I;V)$ and hence $w\in H^1(I;V)$. Moreover,
a direct computation yields $w(0)=0$ and $w'(t)+w(t)=h(t)$ for $t\in I$, which ensures that
$w$ satisfies (\ref{Equ3}). Therefore,
we can conclude that the method of continuity applies and Theorem~\ref{existence1.5} holds for $u_0=0$. 

For $u_0\ne 0$, we let $w\in H^1(I;V)$ such that $w(0)=u_0$ and write
$f^*=\mathcal{A}_1w+\mathcal{A}_2 w'$. Clearly, $f^*\in L^2(I;V')$.
Then the proof above for $u_0=0$ confirms 
the existence of a unique function $v\in H^1(I;V)$ such that
$v(0)=0$ and
\begin{equation*}
\begin{array}{rlll}
\mathcal{A}_1v+\mathcal{A}_2v'=f-f^*.
\end{array}
\end{equation*}
Therefore, the function $u:=w+v$ is the desired weak solution.
\end{proof}


\subsection{Regularity of the solutions to the interface problem}
In this subsection we consider the regularity of the weak solution for (\ref{Equ}), which is important
not only for its theoretical interest but also for the subsequent numerical analysis.
Of our prime concern in this paper is the case when the coefficients $\sigma(x)$ and $\varepsilon(x)$ are discontinuous, 
and the conductivity distribution $\sigma(x)$ in (\ref{Equ}) is unnecessary to be bounded 
below strictly positively.  This feature is common to biological applications. 
Due to the (possibly sharp) jumps of $\sigma(x)$ and $\varepsilon(x)$ across the medium interface, 
the solution to (\ref{Equ}) does not expect a desired global regularity like $H^2(\Om)$, but it is shown 
in this section that this $H^2$-regularity is true locally in each medium region $\Om_i$ for $i=1$, $2$. And 
such local $H^2$-regularity 
is proved sufficient in the next section for us to establish the desired optimal convergence order for finite element approximations. 
As we are not aware of proofs of such local regularities for time-dependent PDEs with large jumps in coefficients 
in literature, even for standard parabolic equations, we will present a rigorous proof here 
for the non-standard time-dependent PDE  (\ref{Equ}). 
%
%
We start with the introduction of some standard assumptions.
\bn
\item[$(A2)$] $\Omega$ consists of  two $C^2$-subdomains $\Omega_1$ and $\Omega_2$ with
$\Omega_1\Subset \Omega$, $\Omega_2:=\Omega\backslash \overline{\Omega}_1$;

\item[$(A3)$] $\varepsilon_i:=\varepsilon|_{{\Omega}_i}$ and
$\sigma_i:=\sigma|_{{\Omega}_i}$ are continuously differentiable in
$\overline{\Omega}_i$ ($i=1,2$).

\en

The interface problem (\ref{Equ}) is often complemented with the following physical 
interface conditions:
\begin{equation}\label{jump}
[u(t)]=0 \quad {\text on} \ I\times \Gamma,
\quad [\sigma \frac{\partial u(t)}{\partial \nu}+\varepsilon
\frac{\partial u'(t)}{\partial \nu}]=0 \quad {\text on} \ I\times\Gamma,
\end{equation}
where $\Gamma:=\p\Omega_1$ is the interface, 
and $[u(t)]:=u_1|_{\Gamma}-u_2|_{\Gamma}$, 
$[\sigma \frac{\partial u(t)}{\partial \nu}+\varepsilon
\frac{\partial u'(t)}{\partial \nu}]:=\sigma_1\f{\p u_1(t)}{\p \nu_1}
+\sigma_2\f{\p u_2(t)}{\p \nu_2}+\varepsilon_1\f{\p u'_1(t)}{\p \nu_1}+\varepsilon_2\f{\p u'_2(t)}{\p \nu_2}$
on $\Gamma$. Here $u_i$ stand for the restrictions of $u$ to $\Omega_i$, and ${\p }/{\p \nu_i}$
denotes the outer normal derivative with respect to $\Omega_i$, $i=1,2$.
To deal with
the interface problem, we introduce a Banach space
$$
\mathcal{Y}=\{u\in V;
u_i\in H^2(\Omega_i), \ i=1,2\}
$$
with the norm
$$
\|u\|_{\mathcal{Y}}=\|u\|_V+\|u_1\|_{H^2(\Omega_1)}+\|u_2\|_{H^2(\Omega_2)}, 
\quad \forall u\in\mathcal{Y}.
$$

\begin{Definition}
Let $f\in L^2(I;H)$ and $u_0\in \mathcal{Y}$.
A function
$u\in H^1(I;\mathcal{Y})$ is called a strong solution of $(\ref{Equ})$ with the jump conditions $(\ref{jump})$
if $u(0)=u_0$ and
the relation
\begin{equation}\label{Def3}
-\nabla\cdot\left(\sigma(x)\nabla u(t,x)+\varepsilon(x)\nabla u'(t,x)\right)=f(t,x)
\end{equation}
holds for a.e. $t\in I$ and a.e. $x\in\Omega_i$ $(i=1,2)$.
\end{Definition}
Before proving the existence of a strong solution to the interface problem, we first establish
the following result.
\begin{Lemma}\label{necessary}
Let $u$ be the weak solution of $(\ref{Def})$. Assume that
 $f\in L^2(I;H)$, $u_0\in \mathcal{Y}$,
 $u\in H^1(I;\mathcal{Y})$, $\p\Omega_1$ and
$\p\Omega_2$ are Lipschitz continuous. Then $u$ is a strong solution for
$(\ref{Equ})$ and $(\ref{jump})$.
\end{Lemma}
\begin{proof}
 We obtain, upon integration by parts, that for a.e. $t\in I$,
\beqn
\int_{\Omega_i}\left(-\nabla\cdot\left(\sigma \nabla u+\varepsilon \nabla u' \right)v-fv\right)dx=
\int_{\Omega_i}\left(\sigma\nabla u\cdot \nabla v+\varepsilon \nabla u\cdot \nabla v-fv\right)dx,  \quad \forall v\in H^1_0(\Omega_i),
\eqn
 which implies that
\begin{equation*}
-\nabla\cdot\left(\sigma(x)\nabla u(t,x)+\varepsilon(x)\nabla u'(t,x)\right)=f(t,x) \
\end{equation*}
holds for a.e. $t\in I$ and a.e. $x\in\Omega_i$ $(i=1,2)$.
It remains to show that the weak solution also satisfies the jump conditions $(\ref{jump})$. By integration by parts
we have for a.e. $t\in I$,
\begin{equation*}
\begin{array}{l}
\displaystyle 0=\int_{\Omega_1\cup \Omega_2 }\left(-\nabla\cdot\left(\sigma \nabla u+\varepsilon \nabla u'\right)v
-fv\right)dx\\
\quad\displaystyle =\int_{\Omega}
\left(\sigma\nabla u\cdot \nabla v+\varepsilon \nabla u\cdot \nabla v-fv\right)dx
-\int_{\Gamma}[\sigma \frac{\partial u}{\partial \nu}+\varepsilon  \frac{\partial u'}{\partial \nu}]vdx,   \quad \forall v\in V.
\end{array}
\end{equation*}
>From this and the definition of weak solutions  it follows that
$$
\int_{\Gamma}[\sigma \frac{\partial u}{\partial \nu}+\varepsilon  \frac{\partial u'}{\partial \nu}]vdx=0,  \q \forall\,v\in V.
$$
The arbitrariness of $v$ shows that $u$ satisfies the second jump condition in (\ref{jump}).
The first condition in  (\ref{jump}) is
a direct consequence of the fact that $u\in H^1(I;V)$. This completes the proof.
\end{proof}

>From the lemma above, we know that the key point is
to get the regularity of the weak solutions, which is the main subject
of the following theorem.

\begin{Theorem}\label{Regularity}
Let $f\in L^2(I;H)$ and
$u_0\in \mathcal{Y}$.
Under the assumptions
$(A1)$, $(A2)$ and $(A3)$, the interface problem $(\ref{Equ})$ and
$(\ref{jump})$
admits a unique strong solution $u$, which satisfies
\begin{equation*}
\|u\|_{H^1(I;\mathcal{Y})}\lesssim \|f\|_{L^2(I;H)}+\|u_0\|_{\mathcal{Y}}.
\end{equation*}
\end{Theorem}
\begin{proof}
>From Corollary \ref{existence1.5}, there exists a weak solution $u\in H^1(I;V)$ to (\ref{Equ}).
In view of Lemma \ref{necessary} and Theorem \ref{prior-estimate}, it suffices to show that $u\in H^1(I;\mathcal{Y})$
and
$$
\|u\|_{H^1(I;\mathcal{Y})}\lesssim \|u\|_{H^1(I;V)}+\|f\|_{L^2(I;H)}+\|u_0\|_{\mathcal{Y}}.
$$

The proof is divided into two parts. We only show that $u|_{\Omega_1}\in H^2(\Omega_1)$,
since the result that $u|_{\Omega_2}\in H^2(\Omega_2)$ can be proven in the same way.
Henceforth we denote by $C$ a generic constant that depends only on the cut-off functions, the final observation time
$T$ and the coefficients $\varepsilon$ and $\sigma$, and is always independent of the size of 
the difference parameter $h$ in \eqref{DQ}.

We first establish the interior regularity of the solution and its desired  estimate.
Let $U\Subset \Omega_1$ and choose a domain $W$ such that $U\Subset W\Subset \Omega_1$.
We then select a cut-off
function $\eta\in C^\infty_0(W)$ such that $\eta\equiv 1$ on $U$ and vanishes outside of $W$.
Now let $|h|>0$ be small, and $e_k$ be the unit coordinate vector
in $x_k$ direction for $k\in \{1,\cdots,n\}$, and define a function $v=-D_k^{-h}(\eta^2D_k^hu')$
(see \eqref{DQ} for the definition of $D^h_k$).
Clearly, we know $v(t)\in H^1_0(\Omega_1)$, hence also to $V$ for $t\in I$.  Now, 
letting $\sigma^h(x)=\sigma(x+he_k)$, $\varepsilon^h(x)=\varepsilon(x+he_k)$ for $x\in W$, substituting 
this $v$ 
into the left-hand side of (\ref{Def}), and integrating it over $I$,  we find that
\begin{equation*}
\begin{array}{l}
A:=\displaystyle\int_I\left( a_1(u(t),v(t))+a_2(u'(t),v(t))\right) dt \\[3mm]
\quad= \displaystyle \int_I \int_{\Omega} \left( D_k^{h}\left(\sigma \nabla u(t)\right) \cdot \nabla \left(\eta^2D_k^hu'(t)\right) +
 D_k^{h}\left(\varepsilon \nabla u'(t)\right) \cdot \nabla \left(\eta^2D_k^hu'(t)\right) \right) dxdt\\[3mm]
\quad =\displaystyle \int_I\int_{W} \left(\varepsilon^h \eta^2 \nabla D_k^{h} u'(t)\cdot \nabla  D_k^hu'(t)
+\sigma^h \eta^2 \nabla D_k^{h} u'(t)\cdot \nabla  D_k^hu(t) \right)dxdt\\[3mm]
\quad\quad  \displaystyle +\int_I\int_{W} \left(  2 \eta D_k^h\varepsilon D_k^hu'(t) \nabla u'(t)\cdot \nabla \eta
+\eta^2 D_k^h\varepsilon \nabla u'(t)\cdot \nabla D_k^hu'(t)+2\varepsilon^h \eta D_k^hu'(t) \nabla D_k^hu'(t)\cdot \nabla \eta \right)dxdt\\[3mm]
\quad\quad \displaystyle +\int_I\int_{W} \left(  2 \eta D_k^h\sigma D_k^hu'(t) \nabla u(t)\cdot \nabla \eta
+\eta^2 D_k^h \sigma \nabla u(t)\cdot \nabla D_k^hu'(t)+2\sigma^h \eta D_k^hu'(t) \nabla D_k^hu(t)\cdot \nabla \eta \right)dxdt\\
\quad=: {\bf( J)_1}+{\bf( J)_2}+{\bf( J)_3}\,.
\end{array}
\end{equation*}

We now estimate ${\bf( J)_1}$, ${\bf( J)_2}$, and $ {\bf( J)_3}$ one by one.  
It is easy to see that
\bb\label{J1}
{\bf( J)_1} =\f{1}{2}\int_{W}\sigma^h \eta^2 |\nabla  D^{h}_k u(T)|^2dx- \f{1}{2}\int_{W}\sigma^h \eta^2 |\nabla  D^{h}_k u(0)|^2dx
+\int_I\int_W\varepsilon^h\eta^2
|\nabla  D^{h}_k u'(t)|^2dxdt.
\ee
We note that there exists a constant $K>0$ such that $|D_k^h\sigma(x)|\leq K$ and
$|D_k^h\varepsilon(x)|\leq K$ for all $x\in W$ and $0<|h|<\f{1}{2} \mathrm{dist}(W,\p\Omega_1)$.
Using Young's inequality and Lemma \ref{DQ1}, we obtain
\bb\label{J2}
|{\bf( J)_2}|\leq\frac{m}{5}\int_I\int_W \eta^2|\nabla D_k^hu'(t)|^2dxdt+C\int_I\int_\Omega |\nabla u'(t)|^2dxdt.
\ee
Similarly, we can derive 
\bb\label{J30}
\begin{array}{l}
\displaystyle |{\bf( J)_3}|\leq \frac{m}{5}\int_I\int_W \eta^2|\nabla D_k^hu'(t)|^2dxdt+ \delta\int_I\int_W \eta^2|\nabla D_k^hu(t)|^2dxdt\\
\displaystyle \quad\quad\quad\quad +C\int_I\int_\Omega\left( |\nabla u(t)|^2+|\nabla u'(t)|^2\right)dxdt, 
\end{array}
\ee
where $\delta$ is a positive constant to be specified later. 
An interplay of Lemmas \ref{DQ1} and \ref{tool1} implies that
$$
\int_{W}\eta|\nabla D_k^hu(t)|^2dx\leq C'(\int_{\Omega}|\nabla D_k u(0)|^2dx+\int_I \int_{\Omega}|\nabla D_k u'(s)|^2dxds),  \quad \forall\,t\in I, 
$$
with some constant $C'>0$, 
whence (\ref{J30}) ensures that
\bb\label{J3}
\begin{array}{l}
\displaystyle |{\bf( J)_3}|\leq \frac{2m}{5}\int_I\int_W \eta^2|\nabla D_k^hu'(t)|^2dxdt
+C\int_I\int_\Omega\left( |\nabla u(t)|^2+|\nabla u'(t)|^2\right)dxdt,
\end{array}
\ee
if $\delta$ is chosen small enough, say $\delta=m/(5TC')$.

On the other hand, using  Young's inequality and Lemma \ref{DQ1} again,
we deduce 
\beqn\label{RSD}
B&:=&\int_I\int_\Omega f(t)v(t)dxdt \notag\\
&\leq &\frac{m}{5}\int_I\int_W \eta^2|\nabla D_k^hu'(t)|^2dxdt+C
\left(\int_I\int_\Omega \left(|f(t)|^2+|\nabla u'(t)|^2\right)dxdt+\|u_0\|_{\mathcal{Y}}\right).\notag
\eqn
Since $A=B$, we combine (\ref{J1}) with (\ref{RSD}) to get 
\begin{equation}\label{R-inequ1}
\begin{array}{rlll}
 &&\displaystyle\f{2m}{5}\int_I\int_W\eta^2
|\nabla  D^{h}_k u'(t)|^2dxdt\\
&\leq&  \displaystyle \frac{m}{5}\int_I\int_W \eta^2|\nabla D_k^hu'(t)|^2dxdt+C(\int_I\int_\Omega \left(|\nabla u(t)|^2+|\nabla u'(t)|^2+|f(t)|^2\right)dxdt+\|u_0\|_{\mathcal{Y}}), 
\end{array}
\end{equation}
which implies 
\beqnx
\sum^n_{i=1}\int_I \|D^h_k D_i u'(t)\|^2_{L^2(U)}
 dt \displaystyle \lesssim  \int_I\left(\|u'(t)\|^2_{V}+\|u(t)\|^2_{V}+\|f(t)\|^2_H \right)dt+\|u_0\|^2_{\mathcal{Y}} 
\eqnx
for all $k=1,2,\cdots, n$ and sufficiently small $|h|\neq 0$.
By applying Lemmas \ref{DQ3}  and  \ref{tool1}, we come to 
\bb\label{R-interior}
\|w\|_{H^1(I;H^2(U))}\lesssim \|f\|_{L^2(I;H)}+\|u\|_{H^1(I;V)}+\|u_0\|_{\mathcal{Y}}.
\ee

Next, we establish the boundary regularity and the desired  estimate.
We first use the standard argument to straighten out the boundary, i.e., flatting
out the boundary by changing the coordinates near a boundary point (cf. \cite[Chap.\,6.2]{Gilbarg}).
Given $x_0\in\p\Omega_1$, there exists a ball $B=B_r(x_0)$ with radius $r$
and a $C^2$-diffeormorphism $\Psi:B \to\Psi(B)\subset\R^n$ such that det$|\nabla \Psi|=1$,
$U':=\Psi(B)$ is an open set,
$\Psi(B\cap\Omega_1)\subset \R_{+}^n$
and $\Psi(B\cap\p\Omega_1)\subset \p\R_{+}^n$, where
$\R_{+}^n$ is the half-space in the new coordinates.
Henceforth we write $y=\Psi(x)=(\Psi_1(x),\cdots,\Psi_n(x))$ for
$x\in B$. Then we have $\{y_n>0; y\in U'\}=\Psi(B\cap\Omega_1)$.
Let $\Phi=\Psi^{-1}$, $B^+=B_{\f{r}{2}}(x_0)\cap\Omega_1$,
$G=\Psi(B_{\f{r}{2}}(x_0))$ and
$G^+=\Psi(B^+)$, then we can see $G\Subset U'$ and $G^+\subset G$.
We shall write $D_iw={\p w}/{\p y_i}$ for $i=1,\cdots, n$, and 
$
w(y)=u(\Phi(y))
$, 
$
\hat{f}(y)=f(\Phi(y))
$
for 
$y\in U'$. 
Now using the transformation function $\Psi$, the original equation on $I\times B$ can be transformed into
an equation of the same form on $I\times U'$, i.e., for a.e. $t\in I$,
\bb\label{T-equ}
\int_{U'}\big(\sum^n_{i,j=1}{\hat{\sigma}_{ij}}D_i w(t) D_j v+\sum^n_{i,j=1}\hat{\varepsilon}_{ij}
D_i w'(t) D_j v \big)dy
=\int_{U'}\hat{f}(t) v dy, \quad \forall v\in H^1_0(U'),
\ee
where the coefficients $\hat{\sigma}_{ij}(y)$ and $\hat{\varepsilon}_{ij}(y)$ are given by 
\bb\label{general-Equ}
\hat{\sigma}_{ij}(y):=\sum^n_{r=1}\sigma(\Phi(y)) \f{\p \Psi_i}{\p x_r} (\Phi(y)) \f{\p \Psi_j}{\p x_r}(\Phi(y)), \
\hat{\varepsilon}_{ij}(y):=\sum^n_{r=1}\varepsilon (\Phi(y))\f{\p \Psi_i}{\p x_r} (\Phi(y)) \f{\p \Psi_j}{\p x_r}(\Phi(y))
\ee
for $1\leq i,j\leq n$ and $y\in U'$.
It is not difficult to see that
$$
\sum^n_{i,j=1}\hat{\varepsilon}_{ij}(y)\xi_i\xi_j\geq m|\xi|^2,
\quad\quad \sum^n_{i,j=1}\hat{\sigma}_{ij}(y)\xi_i\xi_j\geq 0, \quad \forall \ (y, \xi)\in U'\times\R^n,
$$

Choosing a domain $W'$ such that $G\Subset W'\Subset U'$,
we then select a cut-off
function, which is still denoted by $\eta$, such that $\eta\equiv 1$ on $G$ and vanishes outside $W'$.
Now let $|h|>0$ be small, and $\hat{e}_k$
be the unit coordinate vector in the $y_k$ direction for $k\in \{1,\cdots,n-1\}$. In the sequel,
$D^h_k$ stands for the difference quotient in the direction $\hat{e}_k$.
We observe that there exists a constant $K'>0$ such that
$|D_k^h\hat{\sigma}_{i,j}(y)|\leq K'$ and $|D_k^h\hat{\varepsilon}_{i,j}(y)|\leq K'$
for a.e. $y\in W'$, all $0<|h|<\f{1}{2} \mathrm{dist}(W', \p U')$ and $1\leq i,j\leq n$.
Then, a natural variant of the reasoning leading to (\ref{R-inequ1})
shows that
\begin{equation*}
\begin{array}{rlll}
&&\displaystyle \frac{2m}{5}\int_I\int_{W'}\left(\sum^n_{i=1}\eta^2(D_k^h D_i  w'(t))^2\right)dydt \\
&\leq& \displaystyle \frac{m}{5}\int_I\int_{W'}\left(\sum^n_{i=1}\eta^2(D_k^h D_i  w'(t))^2 \right)dydt +C\int_I(\| w(t)\|^2_{H^1(U')}+\|w'(t)\|^2_{H^1(U')}
+\|\hat{f}(t)\|_{L^2(U')})dt\\
&&\displaystyle+C\|w(0)\|_{H^2(U'_{-}\cup U'_{+})},
\end{array}
\end{equation*}
where $\|w(0)\|_{H^2(U'_{-}\cup U'_{+})}:=\|w(0)\|_{H^2(U'_{-})}+\|w(0)\|_{H^2(U'_{+})}$ with
$U'_{+}=U'\cap \R_{+}^n$ and $U'_{-}=U'\backslash\ \overline{U'_{+}}$.
We can derive from the resulting inequality that
\beqnx
&&\displaystyle \sum^n_{i=1}\int_I
\| D^h_k D_i w'(t)\|^2_{L^2(G^+)}  dt\\
\displaystyle &\lesssim & \displaystyle \int_I\left(\|w'(t)\|^2_{H^1(U')}+\|w(t)\|^2_{H^1(U')}
+\|\hat{f}(t)\|^2_{L^2(U')} \right)dt+\|w(0)\|_{H^2(U'_{-}\cup U'_{+})} 
\eqnx
for $k=1,\cdots, n-1$ and all sufficiently small $|h|\neq 0$, where we have also used
the fact $\eta=1$ on $G^+$. Using Lemma \ref{DQ3}, we have
\bb\label{R-inequ3}
\sum_{1\leq i,j<2n}\|D_{i,j} w\|_{H^1(I;L^2(G^+))}\lesssim \|\hat{f}\|_{L^2(I;L^2(U'))}+\|{w}\|_{H^1(I;H^1(U'))}+\|w(0)\|_{H^2(U'_{-}\cup U'_{+})},
\ee
where $D_{i,j}w=D_iD_jw$.
>From (\ref{T-equ}) we obtain upon integration by parts that
for a.e. $t\in I$,
\begin{equation*}
\begin{array}{l}
\displaystyle\int_{G^+}
\hat{\sigma}_{nn}D_n w(t) D_n \varphi+\hat{\varepsilon}_{nn} D_nw'(t) D_n \varphi dy \\
\displaystyle=\int_{G^+}\left(\hat{f}(t) +\sum_{1\leq i,j<2n} D_i(\hat{\varepsilon}_{ij} D_j w'(t))+
\sum_{1\leq i,j<2n} D_i(\hat{\sigma}_{ij} D_j w(t))\right)\varphi dy\\
\end{array}
\end{equation*}
for any $\varphi\in C^\infty_0(G^+)$. Noting that $\hat{\sigma}_{ij}$ and
$\hat{\varepsilon}_{ij}$ are both continuously differentiable in $\overline{G}^+$ and the estimate (\ref{R-inequ3}),
 the right-hand side of the
equation above is well-defined and so we find that for a.e. $t\in I$, the weak
derivative of $\hat{\sigma}_{nn}D_n w(t)+ \hat{\varepsilon}_{nn} D_nw'(t)$ with
respect to
$y_n$ exists and it satisfies
\bb\label{Regular-identiy0}
-D_n \left(\hat{\sigma}_{nn}D_n w(t)+ \hat{\varepsilon}_{nn} D_nw'(t)\right)=
\hat{f}(t) +\sum_{1\leq i,j<2n} D_i(\hat{\varepsilon}_{ij} D_j w'(t))+
\sum_{1\leq i,j<2n} D_i(\hat{\sigma}_{ij} D_j w(t)) . 
\ee
For the sake of brevity, we write $g:=\hat{\varepsilon}_{nn}HD_n w$, where
$$
H(t,y):=\exp(\f{\hat{\sigma}_{nn}(y)}{\hat{\varepsilon}_{nn}(y)}t)\quad (t,y)\in I\times G^+.
$$
It follows readily that $H$ is strictly positive and $H\in C^1(I\times G^+)$.
(\ref{Regular-identiy0}) ensures that for a.e. $t\in I$,
\bb\label{Regular-identity}
-D_n\left(\f{g'(t)}{H}\right)=\hat{f}(t) +\sum_{1\leq i,j<2n} D_i(\hat{\varepsilon}_{ij} D_j w'(t))+
\sum_{1\leq i,j<2n} D_i(\hat{\sigma}_{ij} D_j w(t)).
\ee
A direct computation yields for a.e. $t\in I$,
\bb\label{Regular-identity2}
D_ng'(t)=\f{D_n H(t)}{H(t)}
g'(t)-H(t)\Big(\hat{f}(t)+\sum_{1\leq i,j<2n} D_i(\hat{\varepsilon}_{ij} D_j w'(t))
+\sum_{1\leq i,j<2n} D_i(\hat{\sigma}_{ij} D_j w(t))\Big).
\ee
Since  $\|g'\|_{L^2(I;L^2(G^+))}\lesssim \|w\|_{H^1(I;H^1(G^+))}$,
 we infer from
(\ref{R-inequ3}) and (\ref{Regular-identity2})  that
\bb\label{g1}
\|D_n g'\|_{L^2(I;L^2(G^+))}\lesssim \|\hat{f}\|_{L^2(I;L^2(U'))}+\|{w}\|_{H^1(I;H^1(U'))}+\|w(0)\|_{H^2(U'_{-}\cup U'_{+})} .
\ee
As $\|D_ng(0)\|\lesssim \|w(0)\|_{H^2(U_{+}')}$, an application of Lemma \ref{tool1} yields
\bb\label{g2}
\|D_n g\|_{L^2(I;L^2(G^+))}\lesssim \|\hat{f}\|_{L^2(I;L^2(U'))}+\|{w}\|_{H^1(I;H^1(U'))}+
\|w(0)\|_{H^2(U'_{-}\cup U'_{+})}\,.
\ee
We can then conclude from (\ref{g1}) and (\ref{g2}) that
$$
\|D_{n,n}w\|_{H^1(I;L^2(G^+))}
\lesssim \|\hat{f}\|_{L^2(I;L^2(U'))}+\|{w}\|_{H^1(I;H^1(U'))}+\|w(0)\|_{H^2(U'_{-}\cup U'_{+})}.
$$
Combining this with estimate (\ref{R-inequ3}), and
 transforming $w$ back to $u$ in the resulting inequality, we find
\bb
\|u\|_{H^1(I;H^2(B^+))}
\lesssim \|{f}\|_{L^2(I;H)}+\|u\|_{H^1(I;V)}+\|u_0\|_{\mathcal{Y}}.
\ee
By choosing a finite set of balls $\{B_{r_i/2}(x_{i})\}^N_{i=1}$ such that
it covers the boundary and then adding the estimates over these balls,
we obtain the desired result.
\end{proof}

Using the standard arguments (cf. \cite[Theorem 3.2.1.2]{Grisvard})
with some natural modifications and the estimates above,
we can prove the following regularity result in a general convex domain.
\begin{Theorem}\label{Regularity3}
Let $f\in L^2(I;H)$ and
$u_0\in \mathcal{Y}$. Assume that $\Omega$ is a bounded and convex domain,  $\Omega_1\Subset\Omega$ a $C^2$ subdomain,
and that $(A1)$ and $(A3)$ hold. Then, the interface problem
$(\ref{Equ})$ and $(\ref{jump})$ admits a unique strong solution, which satisfies
\bb
\|u\|_{H^1(I;\mathcal{Y})}\lesssim \|f\|_{L^2(I;H)}+\|u_0\|_{\mathcal{Y}}.
\ee
\end{Theorem}

\subsection{Existence of a strong solution for smooth coefficients}

For the case with smooth coefficients, if we use
\bn
\item[$(A4)$] $\partial\Omega$ is $C^2$ and $\sigma,\varepsilon\in C^{1}(\overline{\Omega})$,
\en
instead of ($A2$) and  ($A3$), then we can obtain a better regularity result as follows, using the same argument as in the poof of Theorem \ref{Regularity}:

\begin{Theorem}\label{Regularity2}
Let $f\in L^2(I;H)$ and $u_0\in \mathcal{X}$. Under  assumptions ($A1$) and ($A4$),
the equation $(\ref{Equ})$ admits a unique strong solution $u$,
which satisfies the following estimate:
\begin{equation*}
\|u\|_{H^1(I;\mathcal{X})}\lesssim \|f\|_{L^2(I;H)}+\|u_0\|_{\mathcal{X}}.
\end{equation*}
\end{Theorem}
\begin{Remark}
By the standard semigroup theory (cf. $\cite{Yagi}$), we can actually  achieve a better estimate, i.e.,  under the assumptions of
Theorem \ref{Regularity2}, we have $u\in C^1([0,T];\mathcal{X})$. That is, $u$ is a classical solution.
\end{Remark}

\section{Finite element approximation and error estimates} \label{sec4}
In this section we propose a fully discrete finite element scheme to approximate the solution
of the interface problem $(\ref{Equ})$ and $(\ref{jump})$, and establish its optimal convergence  
under the minimum regularity assumptions on the given data. 
To do so, we first consider an auxiliary semi-discrete finite element scheme  for 
the concerned interface problem and develop its optimal convergence, which 
will lead to the optimal convergence of the fully discrete scheme. 

Unless otherwise notified, we assume below that
$f\in L^2(I;H)$ and $u_0\in \mathcal{Y}$.
For the sake of exposition, we further make the following
assumptions:
\bn
\item[$(A5)$]  $\Omega$ is a convex polygon or polyhedron in $\mathbb{R}^n$ with $n=2$ or $3$, and $\Omega_1\Subset \Omega$
is a domain with $C^2$-boundary;
\en

\bn
\item[$(A6)$]  The coefficients $\varepsilon$ and $\sigma$ are constants in each domain, namely,
$\varepsilon=\varepsilon_i$ and $\sigma=\sigma_i$ in $\Omega_i$, $i=1,2$, where $\varepsilon_i$ and $\sigma_i$
are two positive constants.
\en
Clear, assumption ($A1$) is satisfied if ($A6$) holds. From Theorem \ref{Regularity3}, it follows that
there exists a strong solution to the interface problem $(\ref{Equ})$ and $(\ref{jump})$.
\begin{Remark}
For the sake of exposition, we assume that
$\Omega$ is a convex polygon (if $n=2$) or a convex polyhedral domain (if $n=3$). The actual curved boundary
can be treated in the same manner as we handle the interface $\Gamma$ in
our subsequent analysis of this section.
\end{Remark}

We now introduce a triangulation of the domain $\Omega$. First we
triangulate $\Omega_1$ using a quasi-uniform mesh $\mathcal{T}^1_h$
with simplicial elements of size $h$, which form a polyhedral domain $\Omega_{1,h}$.
The triangulation $\mathcal{T}^1_h$ is done such that
all the boundary vertices of $\Omega_{1,h}$ lie on the boundary
of $\Omega_1$. Then we triangulate $\Omega_2$ using a quasi-uniform mesh $\mathcal{T}^2_h$
with simplicial elements of size $h$,
which form a polyhedral domain $\Omega_{2,h}$. The triangulation $\mathcal{T}^2_h$ is done
such that  all the vertices
of the outer polyhedral boundary $\p\Omega$ are also the vertices of $\Omega_{2,h}$,
while all the vertices on the inner boundary of $\Omega_{2,h}$ match the
boundary vertices of $\Omega_{1,h}$.
More precisely, the triangulation $\mathcal{T}_h$ satisfies the following conditions:
\bn
\item[$(T1)$]  $\overline{\Omega}=\cup_{K\in \mathcal{T}_h}K$;

\item[$(T2)$]  if $K_1,K_2\in\mathcal{T}_h$ with $K_1\neq K_2$, then either
$K_1\cap K_2=\emptyset$ or $K_1\cap K_2$ is a common vertex, an edge or a face;

\item[$(T3)$] for each $K$, all its vertex is completely contained in
either $\overline{\Omega}_1$ or $\overline{\Omega}_2$.

\en

Now we define $V_h$ to be the continuous piecewise linear finite element
space on the triangulation $\mathcal{T}_h$ and $V^0_h$ the
closed subspace of $V_h$ with its functions vanishing
on the boundary $\p\Omega$. Then, we study the approximation of
piecewise smooth functions by  finite elements in $V_h$.
Clearly, the accuracy of this approximation depends
on how well the mesh $\mathcal{T}_h$ resolve the interface $\Gamma$. Following the notation
introduced in
\cite{Li}, we define, for $\lambda>0$ with $\lambda<\min\{ \mathrm{dist}(\Gamma,\p \Omega), \f{h}{2}\}$, a
tubular neighborhood $S_\lambda$ of
$\Gamma$ by
$$
S_\lambda:=\{x\in\Omega;\ \mathrm{dist}(x,\partial\Gamma)<\lambda\}.
$$
Then, we decompose $\mathcal{T}_h$ into three disjoint subsets $\mathcal{T}_h=\mathring{\mathcal{T}}_h^1\cup \mathring{\mathcal{T}}_h^2\cup \mathcal{T}_{*}$, where
$$
\mathring{\mathcal{T}}_h^i=\{K\in\mathcal{T}_h;K\subset \Omega_i\backslash S_\lambda\}, \quad i=1,2,
$$
and $\mathcal{T}_{*}:=\mathcal{T}_h\backslash (\mathring{\mathcal{T}}_h^1\cup \mathring{\mathcal{T}}_h^2)$. Furthermore, we write
$\mathcal{T}^i_{*}=\{K\in \mathcal{T}_{*}; K\subset\Omega_i\cup S_\lambda \}$.  Since
$\Gamma$ is of class $C^2$,
we know  from \cite{Chen,Feistauer} that there exists $\lambda>0$ such that
\bb\label{lambda}
\lambda=O(h^2),
\ee
and  $\mathcal{T}_*=\mathcal{T}^1_{*}\cup \mathcal{T}^2_{*}$ and $\mathcal{T}^1_{*}\cap \mathcal{T}^2_{*}=\emptyset$,
provided that $h$ is appropriately small.
An important observation is that
$\overline{\Omega}_{i,h}=\cup\{\overline{K}; K\in\mathring{\mathcal{T}}^i_h\cup \mathcal{T}_{*}^i\}$, $i=1,2$, i.e.,
$\mathcal{T}_h^i=\mathring{\mathcal{T}}^i_h\cup \mathcal{T}_{*}^i$.
The notation $S_\lambda$ not only quantifies how well
the mesh $\mathcal{T}_h$ resolves the interface, but it also allows
us to use the lemma \ref{deltaestimate}-\ref{wh}, which were first established
in \cite{Li},  in the subsequent analysis.

We note that the evaluation of the entries of the
stiffness matrix involving interface elements is not trivial in
the three-dimensional case if the mesh is not aligned with the interface. So we shall adopt
the following more convenient approximation bilinear forms
$a_{i,h}(\cdot,\cdot):V\times V\to \mathbb{R}$:

$$
a_{1,h}(u,v):=\sum^2_{i=1}\int_{\Omega_{i,h}}\sigma_i \nabla u\cdot \nabla v dx\, \q\m{and} \q 
a_{2,h}(u,v):=\sum^2_{i=1}\int_{\Omega_{i,h}}\varepsilon_i \nabla u\cdot \nabla v dx.
$$

To approximate the problem in space optimally,
we introduce the projection operator $Q_h:\mathcal{Y}\cap V\to V_h^0$. For each
$u\in\mathcal{Y}$, let $f^{*}=-\varepsilon_i\Delta u_i$ in $\Omega_i$,
$i=1,2$, and
$g^*=[\varepsilon\f{\p u}{\p \nu} ]$.
Clearly, $f^*\in H$ and $g^*\in L^2(\Gamma)$.
Then, we can define $Q_h:\mathcal{Y}\cap V\to V_h^0$ by
$$
a_{2,h}(Q_hu,v_h)=(f^*,v_h)+ \langle g^*,v_h \rangle , \quad \forall v\in V_h^0,
$$
where $\langle \cdot ,\cdot \rangle$ denotes the scalar product in $L^2(\Gamma)$.
We note that the right-hand side $L(\cdot):=(f^*,\cdot)+\langle g^*,\cdot  \rangle$
is independent of $h$. Thus, we can
follow the proof of \cite[Theorems 4.1 and 4.8]{Li}, which
mainly focuses on the case when $g^*=0$,
 to obtain the following result.
\begin{Lemma}\label{Qestimate}
We have
\bb\label{Projection}
a_2(u,v_h)=a_{2,h}(Q_hu,v_h), \quad \forall  v_h\in V_h^0.
\ee
Moreover, for any $u\in \mathcal{Y}$, the
following error estimate holds:
$$
\|u-Q_hu\|_{H}+h\|u-Q_hu\|_{V}\lesssim h^2\|u\|_{\mathcal{Y}}.
$$
\end{Lemma}
Now, we present some auxiliary results.
For the difference between the bilinear form
$a_{i}(\cdot,\cdot)$ and its approximated
bilinear form $a_{i,h}(\cdot,\cdot)$, we have the following result (cf. \cite[p. 27]{Li}).
\begin{Lemma}\label{deltaestimate}
Both $a_{1,h}(\cdot,\cdot)$ and  $a_{2,h}(\cdot,\cdot)$ are bounded, and $a_{2,h}(\cdot,\cdot)$ is coercive. Moreover,
the form $a^{\Delta}_{i,h}(u,v):=a_{i}(u,v)-a_{i,h}(u,v)$, $i=1,2$, satisfies
$$
|a_i^{\Delta}(u,v)|\lesssim | u|_{H^1(S_\lambda)}| v|_{H^1(S_\lambda)}.
$$
\end{Lemma}

To estimate the energy-norm and the $L^2$-norm of a
function over $S_\lambda$, we will frequently use the following result (cf. \cite[Lemma 2.1 and Remark 4.2]{Li}).
\begin{Lemma}\label{interpolation}
For any $u\in V$, we have
\bb\label{interp2}
\|u\|^2_{L^2(S_\lambda)}\lesssim \lambda \|u\|^2_{V}.
\ee
Moreover, for any $u\in \mathcal{Y}$,
\bb\label{interp1}
| u|_{H^1(S_\lambda)}^2\lesssim \lambda\|u\|_{\mathcal{Y}}^2
\ee
with $|\cdot |_{H^1(S_\lambda)}$ being the $H^1$-semi norm.
\end{Lemma}
The following estimate is critical for proving our main result (cf. \cite[Lemma \ 4.5]{Li}).
\begin{Lemma}\label{wh}
There exists a positive constant $\mu$ independent of $h$ such that
$$
\|w_h\|_{H^1(S_\lambda)}\lesssim\sqrt{\frac{\lambda}{h}}\|w_h\|_{H^1(S_{\mu h})}, \quad \forall
w_h\in V_h.
$$
\end{Lemma}

\subsection{Semi-discrete finite element approximation and error estimates}

We now consider an auxiliary semi-discrete finite element scheme  for 
our concerned interface problem $(\ref{Equ})$ and $(\ref{jump})$ and develop its optimal convergence, which 
will lead directly to the optimal convergence of the fully discrete scheme in section 4.2.

\medskip
\noindent {\bf Problem ($\mathbf{ P}_{h}$)}.
Let $u_h(0)=Q_h u_0$. Find $u_h\in H^1(I;V_h^0)$ such that for a.e. $t\in I$,
\begin{equation}\label{SD}
a_{1,h}(u_h(t),v_h)+a_{2,h}(u'_h(t),v_h)=\langle f(t),v_h\rangle_{V'\times V},  \quad \forall v_h\in V_h^0.
\end{equation}

We first establish an auxiliary lemma, which will be used in the proof below.
\begin{Lemma}\label{discrete-bound}
Let $f\in L^2(I;V')$ and $u_h$ be the solution to  Problem {\bf(${\bf P_h}$)}.  We have
$$
\|u_h\|_{H^1(I;V)}\lesssim \|f\|_{L^2(I;V')}+\|Q_hu_0\|_{V}.
$$
\end{Lemma}
 Since the proof of this Lemma is the same as that of Theorem \ref{prior-estimate}, we omit the
 details.  With the results above, we are now in a position to prove the error estimate
 in the energy norm.
\begin{Theorem}\label{H1esti} Let
$u$ be the solution to the interface problem $(\ref{Equ})$ and $(\ref{jump})$ and
 $u_h$  the solution to Problem  {\bf ($\mathbf{P}_{h}$)},
then the following estimate holds:
$$
\|u-u_h\|_{H^1(I;V)}
\lesssim h \left(\|f\|_{L^2(I;H)}+\|u_0\|_{\mathcal{Y}}\right).
$$
\end{Theorem}
\begin{proof}
We first have the following decomposition:
\begin{equation}
\begin{array}{l}
\displaystyle \left(\int^T_{0}\big(\|u(t)-u_h(t)\|^2_{V}+\|u'(t)-u'_{h}(t)\|^2_{V}\big)dt
\right)^{\f{1}{2}}\\
\displaystyle \leq \left(\int^T_{0}\big(\|u(t)-Q_h u(t)\|^2_{V}+\|u'(t)-Q_hu'(t)\|^2_{V}\big)dt
\right)^{\f{1}{2}} \\
\quad\quad \displaystyle +\left(\int^T_{0}\big(\|Q_h u(t)-u_h(t)\|^2_{V}
+\|Q_h u'(t)- u'_{h}(t)\|^2_{V}\big)dt\right)^{\f{1}{2}}\\
=: \text{(I)}_1+\text{(I)}_2 .
\end{array}
\end{equation}
Using Lemma \ref{Qestimate} and Theorem \ref{Regularity3}, we obtain
$$
\text{(I)}_1 \lesssim  h
 \|u\|_{H^1(I;\mathcal{Y})}\lesssim h \left(\|f\|_H+\|u_0\|_{\mathcal{Y}}\right).
$$
It suffices to prove a similar estimate for $\text{(I)}_2$.
To this end, we first notice that the function
$w:=u_h-Q_h u$ belongs to $H^1(I;V_h^0)$. In addition,
using the identity that $ (Q_hu)'(t)=Q_h u'(t)$ for a.e. $t\in I$ and the definition
of $u$ and $u_h$, we find  for a.e. $t\in I$,
$$
a_{1,h}(w(t),v_h)+a_{2,h}(w'(t),v_h)=\langle F(t),v_h\rangle_{V'\times V}, \quad  \forall v_h\in V_h^0,
$$
where $F(t)\in V'$ for $t\in I$,  defined by
$$
\langle F(t),v \rangle_{V'\times V}:=a_{1}(u-Q_h u,v )+a_{2}(u'-Q_h u',v )
+a^{\Delta}_{1}(Q_h u,v )+a^{\Delta}_{2}(Q_h u' ,v ), \quad \forall v \in V.
$$
Analogously to Lemma \ref{discrete-bound}, we derive
\bb\label{CI2}
\text{(I)}_2=\|w\|_{H^1(I;V)}\lesssim  \|F\|_{L^2(I;V')}.
\ee
Thus, it remains to estimate $\|F\|_{L^2(I;V')}$. For $t\in I$ and any $v\in V$, we use Lemma
\ref{deltaestimate} to obtain
\begin{equation*}
\begin{array}{l}
\displaystyle |\langle F(t),v\rangle_{V'\times V}|\\
\displaystyle \lesssim \left(\|u(t)-Q_h u(t)\|_{V}+|Q_h u(t)|_{H^1(S_\lambda)}
+\|u'(t)-Q_h u'(t)\|_{V}+|Q_h u'(t)|_{H^1(S_\lambda)}\right)\|v\|_V,
\end{array}
\end{equation*}
which, together with the estimates
$$
|Q_h u(t)|_{H^1(S_\lambda)}\leq |u(t)|_{H^1(S_\lambda)}+|u(t)-Q_h u(t)|_{H^1(S_\lambda)},
$$
and
$$
|Q_h u'(t)|_{H^1(S_\lambda)}\leq |u'(t) |_{H^1(S_\lambda)}
+|u'(t) -Q_h u'(t)|_{H^1(S_\lambda)},
$$
implies that
$$
\|F\|_{L^2(I;V')}\lesssim \|Q_hu-u\|_{H^1(I,V)}+\|u\|_{H^1(I;H^1(S_\lambda))}.
$$
Now Lemmas \ref{Qestimate} and \ref{interpolation}, together with Theorem \ref{Regularity3},  yield
$$
\|F\|_{L^2(I;V')}\lesssim (h+\sqrt{\lambda})\left(\|f\|_{L^2(I;H)}+\|u_0\|_{\mathcal{Y}}\right).
$$
>From this, (\ref{lambda}) and (\ref{CI2}), the desired estimate for $\text{(I)}_2$ is established.
\end{proof}

Now, we are in a position to prove the $L^2$-estimate.

\begin{Theorem}\label{L2esti}
We have the following estimate in $L^2$-norm:
\begin{equation*}
\|u-u_h\|_{L^2(I;H)}\lesssim h^2\left(\|f\|_{L^2(I;H)}+\|u_0\|_{\mathcal{Y}}\right).
\end{equation*}
\end{Theorem}

\begin{proof}
For the duality argument, we define
$w\in H^1(I;V)$ and  $w_h\in H^1(I;V_h^0)$
such that for a.e. $t\in I$,
\begin{equation*}
\begin{array}{l}
a_1(w(t),v)-a_2(w'(t),v)=(u(t)-u_h(t),v), \quad\forall v\in V, \\
a_{1}(w_h(t),v)-a_{2}(w'_{h}(t),v)=(u(t)-u_h(t),v),  \quad \forall v\in V_h^0,
\end{array}
\end{equation*}
which satisfies $w(T)=w_h(T)=0$.
That is, $w^{*}(t):=w(T-t)$ is the weak solution of (\ref{Equ})
with initial value $w^{*}(0)=0$ and $f$ replaced by $u-u_h$. Then Theorem \ref{Regularity3}
implies  that \begin{equation}\label{wregularity}
\displaystyle \|w\|_{ H^1(I;\mathcal{Y})}\lesssim \|u-u_h\|_{L^2(I;H)}.
\end{equation}
Using the same argument employed in Theorem \ref{H1esti} with
a natural modification, we find that
\begin{equation}\label{ceas}
\begin{array}{l}
\displaystyle \|w-w_h\|_{H^1(I;V))}
\displaystyle \lesssim h \|u-u_h\|_{L^2(I;H)}.
\end{array}
\end{equation}
By integration by parts with respect to the time variable, identity (\ref{Projection}) and taking advantage of the
Galerkin orthogonality for $w-w_h$ and $e:=u-u_h$, we know that
\begin{equation}\label{dual1}
\begin{array}{ l}
\displaystyle \int^T_{0}\big(a_1(e,w_h)-a_2(e,w'_{h})\big)dt\\
\displaystyle =\int_0^T\big(a_1(e,w_h)+a_2(e',w_h)\big)dt+a_2(u(0)-Q_h(0),w_h(0))
\\ \displaystyle = \int^T_{0}\left(-a^\Delta_1(u_h,w_h)-a^\Delta_2(u'_{h},w_h)\right)dt
+a^{\Delta}_2(Q_h u(0),w_h(0)),
\end{array}
\end{equation}
and for a.e. $t\in I$,
\begin{equation}\label{dual2}
\displaystyle a_1(w(t)-w_h(t),v)-a_2(w'(t)-w'_{h}(t),v)=0,   \quad \forall v\in V_h^0.
\end{equation}
Applying (\ref{dual1}) and (\ref{dual2}) and integrating by parts with respect to time variable, we obtain
\begin{equation*}
\begin{array}{lllr}
&&\|e\|_{L^2(I;H)}^2\\
 &=& \displaystyle \int^T_{0}\big(a_1(e,w)-a_2(e,w')\big)dt\\
 &=& \displaystyle  \int_0^T\big( a_1(e,w-w_h)-a_2(e,w'-w'_{h}) +a_1(e,w_h)-a_2(e,w'_{h})\big)dt\\\
&=&\displaystyle  \int^T_{0}\big(a_1(u-Q_h u,w-w_h)-a_2(u-Q_h u,w'-w'_{h})\big)dt  \\
 \displaystyle
&& \displaystyle  -\int^T_{0}\big(a^\Delta_1(u_h,w_h)+a^\Delta_2(u'_{h},w_h)\big)dt+a^{\Delta}_2(Q_h u(0),w_h(0))\\\\
 &=& \displaystyle   \int^T_{0}\big(a_1(u-Q_h u,w-w_h)+a_2(u'-Q_h u',w-w_{h})\big)dt  \\
&& \displaystyle  -\int^T_{0}\left(a^\Delta_1(u_h,w_h)+a^\Delta_2(u'_{h},w_h)\right)dt\\
 && \displaystyle  +a^{\Delta}_2(Q_h u(0),w_h(0))+a_2(u(0)-Q_h u(0),w(0)-w_h(0))\\
&=:& \text{(II)}_1+\text{(II)}_2+\text{(II)}_3.
\end{array}
\end{equation*}
By the Cauchy-Schwarz inequality it follows that
\begin{equation*}
\begin{array}{l}
\text{(II)}_1\lesssim \|u-Q_h u\|_{L^2(I;V)}\|w-w_h\|_{L^2(I;V)}+
\|u'-Q_h u'\|_{L^2(I;V)}\|w-w_h\|_{L^2(I;V)}
\end{array}
\end{equation*}
Applying Lemma \ref{Qestimate}, the regularity estimate (\ref{wregularity}) and
Theorem \ref{H1esti}, we have
\bb\label{I1}
 \text{(II)}_1\leq h^2\left(\|f\|_{L^2(I;H)}+\|u_0\|_{\mathcal{Y}}\right)\|e\|_{L^2(I;H)}.
\ee
By Lemma \ref{deltaestimate} and the Cauchy-Schwarz inequality,
\begin{equation*}
\begin{array}{l}
 \text{(II)}_2 \lesssim (|u_h|_{L^2(I;H^1(S_\lambda))}+|u'_h|_{L^2(I;H^1(S_\lambda))})|w_h|_{L^2(I;H^1(S_\lambda))}.
\end{array}
\end{equation*}
Before we further estimate $\text{(II)}_2$, we first bound $u_h$, $w$
and $w_h$ in $H^1(I;H^1(S_\lambda))$. Applying Lemma \ref{interpolation} and Theorem \ref{H1esti}, we have
\begin{equation}\label{I31}
\|u_h\|_{H^1(I;H^1(S_\lambda))}\leq \|e\|_{H^1(I;H^1(S_\lambda))}+\|u\|_{H^1(I;H^1(S_\lambda))}
\lesssim (\sqrt{\lambda}+h)\|u\|_{H^1(I;\mathcal{Y})}.
\end{equation}
On the other hand, using Lemma \ref{interpolation} and the regularity estimate (\ref{wregularity}), it follows that
$$
\|w\|_{H^1(I;H^1(S_\lambda))}\lesssim \sqrt{\lambda}
\|w\|_{H^1(I;\mathcal{Y})}\lesssim \sqrt{\lambda}\|e\|_{L^2(I;H)}.
$$
Using Lemmas \ref{Qestimate} and \ref{wh}, the  regularity estimate (\ref{wregularity}), (\ref{ceas}) and
the condition $2\lambda\leq h$, we have
\beqn\label{I32}
\|w_h\|_{H^1(I;H^1(S_\lambda))}&\lesssim&
\sqrt{\lambda}h^{-\frac{1}{2}}\|w_h\|_{H^1(I;H^1(S_{\mu\lambda}))}\notag\\
&\lesssim& \sqrt{\lambda}h^{-\frac{1}{2}}\left(\|w-w_h\|_{H^1(I;H^1(S_{\mu\lambda}))}+\|w\|_{H^1(I;H^1(S_{\mu\lambda}))}\right)\\
&\lesssim& \sqrt{\lambda}h^{-\frac{1}{2}}\left(\|w-w_h\|_{H^1(I;H^1(S_{\mu\lambda}))}
+h^{\frac{1}{2}}\|w\|_{L^2(I;\mathcal{Y})}\right)\notag\\
&\lesssim& \sqrt{\lambda}\|e\|_{L^2(I;H)}. \notag
\eqn
Now, (\ref{I31}), (\ref{I32}) and Theorem \ref{Regularity3} yield
\begin{equation}\label{I2}
\begin{array}{l}
 \text{(II)}_2 \lesssim (\lambda+\sqrt{\lambda}h)\left(\|f\|_{L^2(I;H)}+\|u_0\|_{\mathcal{Y}}\right)\|e\|_{L^2(I;H)}.
\end{array}
\end{equation}
To bound $\text{(II)}_3$, we first need to estimate $ |Q_hu(0)|_{H^1(S_\lambda)}$ and
$|w_h(0)|_{H^1(S_\lambda)}$. To this end, applying Lemmas \ref{Qestimate}
and \ref{interpolation}, we find that
 $$
 |Q_hu(0)|_{H^1(S_\lambda)}\leq  |Q_hu(0)-u(0)|_{H^1(S_\lambda)}+|u(0)|_{H^1(S_\lambda)}\lesssim  (\sqrt{\lambda}+h)\|u(0)\|_{\mathcal{Y}}.
 $$
On the other hand, from (\ref{embeding})  it readily follows  that
$$
\|w_h(0)\|_{H^1(S_\lambda)}\leq \sup_{t\in \overline{I}}\|w_h(t)\|_{H^1(S_\lambda)} \lesssim \|w_h\|_{H^1(I;H^1(S_\lambda))}.
$$
Similarly, we have $\|w(0)-w_h(0)\|_V \lesssim  \|w-w_h\|_{H^1(I;V)}$. We thus find
$$
\|w_h(0)\|_{H^1(S_\lambda)} \lesssim \|w_h\|_{H^1(I;H^1(S_\lambda))}\lesssim \sqrt{\lambda}\|e\|_{L^2(I;H)}.
$$
Summarizing the above estimates, we get
\beqn\label{I3}
 \text{(II)}_3 &\lesssim& |Q_h u(0)|_{H^1(S_\lambda)}|w_h(0)|_{H^1(S_\lambda)}+\|Q_h u(0)-u(0)\|_V\|w(0)-w_h(0)\|_V\\
&\lesssim&
 (\lambda+h\sqrt{\lambda}+h^2) \|u(0)\|_{\mathcal{Y}}\|e\|_{L^2(I;H)}. \notag
\eqn
Taking (\ref{lambda}), (\ref{I1}), (\ref{I2}), and (\ref{I3}) into
consideration, we can conclude the desired estimate.
\end{proof}

%

\subsection{Fully discrete finite element scheme and error estimates}

In this subsection, we are going to formulate a fully discrete finite element scheme to approximate
the solution to the interface problem (\ref{Equ}) and (\ref{jump}). 
We shall use the backward Euler scheme for the time discretization.
Let us start with dividing the time interval $I$ into
$N$ equally spaced subintervals and using the following nodal
points:
$$
0=t^0<t^1<\cdots <t^N=T,
$$
where $t^n=n\tau$ for
$n=0,1,\cdots, N$ and  $\tau=T/N$.  For any given discrete time
sequence $\{u^n\}_{n=0}^N$ in $V$ and
a function $g(x,t)$ which is continuous with respect to $t$, we can define
$$
\p_\tau w^n=\f{w^n-w^{n-1}}{\tau},  \quad
\overline{g}^n=\f{1}{\tau}\int_{t^{n-1}}^{t^n}g(\cdot, s)ds,\quad \widehat{g}^n(\cdot)=g(\cdot,t^n), \quad n=1,\cdots, N.
$$

Now, we propose  a fully discrete finite element scheme to approximate the solution to the interface problem
(\ref{Equ}) and (\ref{jump}).

\medskip

\noindent {\bf Problem ($\mathbf{P_{h,\tau}}$).}
Let $u^0_h=Q_h u_0$. For each $n=1,2,\cdots, N$, find $u^n_h\in V^0_h$ such that
\begin{equation}\label{FD}
a_{1,h}(u^{n}_h,v_h)+a_{2,h}(\partial_\tau u^n_h,v_h)=( \widehat{f}^n,v_h), \quad \forall v_h\in V_h^0.
\end{equation}

For a discrete sequence $\{u_h^n\}^N_{n=1}$ defined in  Problem {\bf ($\mathbf{P_{h,\tau}}$)},
we can introduce a piecewise constant function in time by
\bb
u_{h,\tau}(\cdot, t)=u_h^n(\cdot),  \quad \forall t\in(t^{n-1},t^n], \quad \ n=1,2,\cdots,N.
\ee
Then, we say that $u_{h,\tau}$ is a solution of  Problem {\bf ($\mathbf{P_{h,\tau}}$)}, which is
a fully discrete approximation of the solution to the interface problem (\ref{Equ}) and (\ref{jump}).
In order to compute the error between $u_{h,\tau}$ and $u$, it suffices to establish
the error between $u_{h,\tau}$ and $u_h$, i.e. the solution of the semi-discrete scheme (\ref{SD}), 
since the error between $u_h$ and
$u$ has been studied in Section 4.1. To this end, we need the following auxiliary result.
\begin{Lemma}\label{difference}
Let $\{F_n\}^N_{n=1}$ be a time discrete sequence lying in $V'$ and $w_h^0=0$.
There exists a unique sequence $\{w^n_h\}^N_{n=1}$ such that for $n=1,2,\cdots, N$,
\begin{equation}\label{wSD}
a_{1,h}(w^{n}_h,v_h)+a_{2,h}(\partial_\tau w^n_h,v_h)=\langle F_n,v\rangle_{V'\t V},  \quad \forall v_h\in V_h^0.
\end{equation}
Moreover, the sequence $\{w^n_h\}^N_{n=1}$ has the following stability estimate:
\bb\label{wh1}
\max_{1\leq n\leq N}\|w_h^{n}\|_V^2\lesssim \tau \sum^N_{n=1} \|F_n\|^2_{V'}.
 \ee
\end{Lemma}

\begin{proof}
The existence and uniqueness follows immediately from the Lax-Milgram theorem. Taking $v_h=2\tau \p w_h^n$ in
(\ref{wSD}), 
then using  the relation
$$
2\tau a_{1,h}( w_h^n,\partial_\tau w^{n}_h)=a_{1,h}(w_h^n,w^n_h)-a_{1,h}(w_h^{n-1},w_h^{n-1})
+\tau^2a_{1,h}(\p_\tau w_h^n,\p_\tau w^n_h), \quad \forall \ n=1,\cdots,N,
$$
and the coercivity of $a_{2,h}(\cdot,\cdot)$, we obtain that
$$
2m\tau\|\p_\tau w_h^n\|^2_V
+a_{1,h}(w_h^{n},w_h^{n})-a_{1,h}(w_h^{n-1},w_h^{n-1})\leq
2\tau  \|F_n\|_{V'}\|\p_\tau w_h^n\|_V , \quad \ \forall n=1,2,\cdots,N.
$$
Adding the above inequalities
from $n=1$ to $n=N$, and
using the Cauchy inequality, one has 
$$
{m}\sum^N_{n=1}\|\p_\tau w_h^n\|^2_V\lesssim \sum^N_{n=1}\|F_n\|_{V'}^2.
$$
Now the desired estimate follows immediately from the inequality  
$$
\|w_h^n\|^2_V\leq T\tau\sum^N_{n=1}\|\p_\tau w_h^n\|^2_V , \quad \forall \ 1\leq n\leq N.
$$
\end{proof}

>From the lemma above we notice that  Problem {\bf ($\mathbf{ P_{h,\tau}}$)} always admits
a unique solution.
\begin{Lemma}\label{deference2}
Let $u_{h,\tau}$ and $u_h$ be the solution of Problem {\bf ($\mathbf{ P_{h,\tau}}$)} and
Problem {\bf  ($\mathbf{ P_{h}}$)}, respectively.
Under the assumption that  $f\in H^1(I;H)$, the following estimate holds:
$$
\|u_h-u_{\tau,h}\|_{L^2(I;V)}\lesssim \tau\left(\|f'\|_{L^2(I;H)}+\|f\|_{L^2(I;H)}+\|u_0\|_{\mathcal{Y}}\right),
$$
\end{Lemma}
\begin{proof}
We first define a piecewise constant function in time such that $u^*_{h,\tau}(0)=Q_hu_0$ and
$$
u^*_{h,\tau}(\cdot,t)= \widehat{u}^n_h(\cdot),   \quad \forall t\in(t^{n-1},t^n], \quad n=1,2,\cdots,N.
$$
Using Lemmas \ref{Qestimate} and \ref{discrete-bound}, it follows readily that
\bb\label{def-eq1}
\|u_h-u^{*}_{h,\tau}\|_{L^2(I;V)}\lesssim \tau \|u_h\|_{H^1(I;V)}\lesssim \tau(\|f\|_{L^2(I;H)}+\|u_0\|_{\mathcal{Y}}).
\ee
Integrating
(\ref{SD}) over $(t^{n-1},t^{n})$ and dividing both sides by $\tau$, we have for $n=1,2,\cdots,N$,
\begin{equation} \label{addeq2}
a_{1,h}(\overline{u}^n_h,v_h)+a_{2,h}(\partial_\tau \widehat{u}^n_h,v_h)=(\o{f}^n,v_h),  \quad\forall v_h\in V_h^0 .
\end{equation}
Subtracting both sides of (\ref{addeq2}) above from those of (\ref{FD}), we can rewrite the resulting
equation as
$$
a_{1,h}(u^{n}_h-\widehat{u}^n_h,v_h)+a_{2,h}(\partial_\tau (u^n_h-\widehat{u}^n_h),v_h)
=(\widehat{f}^n-\o{f}^n,v_h)+a_{1,h}(\overline{u}^n_h-\widehat{u}_h,v_h),  \quad \forall v_h\in V_h^0.
$$
The right-hand side of the equation above defines a functional on $V$ for each $n=1,2,\cdots, N$.
Indeed, we have for $n=1,2,\cdots,N$,
$$
|(\widehat{f}^n-\o{f}^n,v)+a_{1,h}(\overline{u}^n_h-\widehat{u}_h,v)|
\lesssim \left(\|\widehat{f}^n-\o{f}^n\|_{H}+\|\overline{u}^n_h-\widehat{u}_h\|_V\right)\|v\|_V , \quad \forall\ v\in V
$$
by using Poincar\'{e}'s inequality.
Therefore we can apply  Lemma {\ref{difference}} to obtain
\beqnx
\|u^{*}_{h,\tau}-u_{h,\tau}\|^2_{L^2(I;V)}&=&\sum^N_{n=1}\tau \|u^{n}_h-\widehat{u}^n_h\|_V^2
\leq  T \max_{1\leq n\leq N}\|u^{n}_h-\widehat{u}^n_h\|_V^2 \\
&\lesssim&\tau \sum^N_{n=1} \left(\|\widehat{f}^n-\o{f}^n\|^2_{H}+\|\overline{u}^n_h-\widehat{u}_h\|^2_V\right)\\
&\lesssim&\tau^2(\|f'\|^2_{L^2(I;H)}+\|u_h\|^2_{H^1(I;V)})\\
&\lesssim&\tau^2( \|f'\|_{L(I;H)}^2+ \|f\|_{L(I;H)}^2+\|u_0\|^2_{\mathcal{Y}}).
\eqnx
Now the desired result follows from the previous estimate,
(\ref{def-eq1}) and the following triangular inequality
$$
\|u_h-u_{h,\tau}\|_{L^2(I;V)}\leq \|u_h-u^{*}_{h,\tau}\|_{L^2(I;V)}+ \|u^{*}_{h,\tau}-u_{h,\tau}\|_{L^2(I;V)}.
$$
\end{proof}
>From Lemma \ref{deference2} and Theorems \ref{H1esti} and \ref{L2esti},
the following result follows immediately.
\begin{Theorem}\label{Ferror}
Let
$u$ be the solution to the interface problem $(\ref{Equ})$ and $(\ref{jump})$ and
$u_{h,\tau}$ the solution to  Problem {\bf($\mathbf{ P}_{h,\tau}$)}.
Under the assumption of Lemma \ref{deference2}, the following estimates hold:
$$
\|u-u_{h,\tau}\|_{L^2(I;V)}\lesssim
\left(\tau+h\right)(\|f\|_{L^2(I;H)}+\|f'\|_{L^2(I;H)}+\|u_0\|_{\mathcal{Y}}),
$$
$$
\|u-u_{h,\tau}\|_{L^2(I;H)}\lesssim
\left(\tau+h^2\right)(\|f\|_{L^2(I;H)}+\|f'\|_{L^2(I;H)}+\|u_0\|_{\mathcal{Y}}).
$$
\end{Theorem}
%

\end{document}